\documentclass[a4paper, 11pt]{article}
\usepackage{amsmath}
\usepackage{amssymb,esint}
\usepackage{amscd}
\usepackage{xspace}
\usepackage{fancyhdr}
\usepackage{color}
\usepackage{srcltx} 

\newtheorem{theorem}{Theorem}[section]

\newtheorem{definition}[theorem]{Definition}

\newtheorem{lemma}[theorem]{Lemma}

\newtheorem{proposition}[theorem]{Proposition}
\newtheorem{remark}[theorem]{Remark}

\newenvironment{proof}[1][Proof]{\textbf{#1.} }{\hfill\rule{0.5em}{0.5em}}
{\catcode`\@=11\global\let\AddToReset=\@addtoreset
\AddToReset{equation}{section}

\AddToReset{theorem}{section}

\def\nc{\newcommand}

\def\lam{\lambda}

 \def\Om{\Omega}

\nc\pa{\partial}

\nc\CC{\mathbb{C}}
\nc\RR{\mathbb{R}}
\nc\QQ{\mathbb{Q}}
\nc\ZZ{\mathbb{Z}}
\nc\NN{\mathbb{N}}

\begin{document}
\title{Existence and regularity estimates for quasilinear equations with measure data: the case $1<p\leq \frac{3n-2}{2n-1}$ }
	\author{
	{\bf Quoc-Hung Nguyen\thanks{E-mail address: qhnguyen@shanghaitech.edu.cn, ShanghaiTech University, 393 Middle Huaxia Road, Pudong,
			Shanghai, 201210, China.} ~ and~ Nguyen Cong Phuc\thanks{E-mail address: pcnguyen@math.lsu.edu, Department of Mathematics, Louisiana State University, 303 Lockett Hall,
			Baton Rouge, LA 70803, USA. }}}
\date{}  
\maketitle
\begin{abstract}
We obtain existence and global regularity estimates for gradients of solutions to quasilinear elliptic equations  with measure data whose prototypes are of the form 
$-{\rm div} (|\nabla u|^{p-2} \nabla u)= \delta\, |\nabla u|^q +\mu$  	in a bounded main $\Om\subset\RR^n$ potentially with non-smooth boundary. Here either $\delta=0$ or $\delta=1$, $\mu$ is 
a finite signed Radon measure  in $\Omega$, and $q$ is of linear or super-linear growth, i.e., $q\geq 1$. Our main concern is to extend earlier results to  the strongly singular case  	$1<p\leq \frac{3n-2}{2n-1}$. In particular, in the case $\delta=1$ which corresponds to a Riccati type equation, we settle the question of solvability that has been raised for some time in the literature. 
\medskip

\medskip

\medskip


\medskip

\end{abstract}   
                  
\tableofcontents
									
 \section{Introduction and main results} 
 
 This paper can be viewed as a continuation of our earlier work \cite{QH4, 55Ph2} in which we studied  gradient regularity  
 of solutions to  quasilinear elliptic equations with measure data
 \begin{eqnarray}\label{5hh070120148}
 \left\{ \begin{array}{rcl}
 -{\rm div}(A(x, \nabla u))&=& \mu \quad \text{in} ~\Omega, \\
 u&=&0  \quad \text{on}~ \partial \Omega,
 \end{array}\right.
 \end{eqnarray} 											
 and applied it to obtain sharp existence results for the  Riccati type equation
 \begin{eqnarray}\label{Riccati}
 \left\{ \begin{array}{rcl}
 -{\rm div}(A(x, \nabla u))&=& |\nabla u|^q + \mu \quad \text{in} ~\Omega, \\
 u&=&0  \quad ~~~~~~~~~~~\text{on}~ \partial \Omega.
 \end{array}\right.
 \end{eqnarray} 																																						
 Here  $\Omega$ is a bounded open subset of $\mathbb{R}^{n}$, $n\geq2$, and $\mu$ is  a finite signed Radon measure  in $\Omega$. The principal operator 
 ${\rm div}(A(x, \nabla u))$ is modeled after the $p$-Laplacian  defined by $\Delta_p u:= {\rm div } (|\nabla u|^{p-2}\nabla u)$. In the papers 
 \cite{55Ph2, 55Ph2-2} and \cite{QH4}  the case $2-\frac{1}{n}<p\leq n$ and the case $\frac{3n-2}{2n-1}<p\leq 2-\frac{1}{n}$ were considered, respectively.  
  In this paper, we consider the remaining `strongly singular'  case $1<p\leq \frac{3n-2}{2n-1}$, which eventually settle the question of solvability (raised in \cite[pages 13--14]{VHV}) for   
  \eqref{Riccati}  for all $1<p\leq n$ and $q\geq 1$.

More precisely, in \eqref{5hh070120148}-\eqref{Riccati}, the nonlinearity  $A:\mathbb{R}^n\times \mathbb{R}^n\to \mathbb{R}^n$ is a Carath\'eodory vector valued function, i.e., $A(x,\xi)$ is measurable in $x$ and continuous with respect to $\xi$ for a.e. $x$. Moreover, for a.e. $x$, $A(x,\xi)$ is continuously differentiable 
in $\xi$ away from the origin and																			satisfies 
                                       \begin{align}
                                       \label{condi1}| A(x,\xi)|\le \Lambda |\xi|^{p-1},\quad | \nabla_\xi A(x,\xi)|\le \Lambda |\xi|^{p-2},
                                       \end{align}
                                        \begin{align}
                                       \label{condi2}  \langle \nabla_\xi A(x,\xi)\eta,\eta\rangle\geq \Lambda^{-1} |\eta|^2|\xi|^{p-2},
                                       \end{align}
                                          for every $(\xi,\eta)\in \mathbb{R}^n\times \mathbb{R}^n\backslash\{(0,0)\}$ and a.e. $x\in \mathbb{R}^n$, where  $\Lambda$ is a  positive constant.

																					As for $p$ in \eqref{condi1}-\eqref{condi2}, we shall restrict ourselves to the range: 
                                          \begin{align*}
                                         1<p\leq  \frac{3n-2}{2n-1}.
                                          \end{align*}     
                                             
We shall also require that $A(x,\xi)$ satisfy a smallness condition of BMO type in the $x$-variable. Such a condition is called the $(\delta, R_0)$-BMO condition defined below (see, e.g., \cite{BW2, KuMi12, 55Ph2}). This  condition allows $A(x, \xi)$ has discontinuity in $x$ and it can be used as a substitute for the Sarason \cite{Sa} VMO  condition.

\begin{definition} 
 We say that $A({x, \xi})$ satisfies a $(\delta, R_0)$-BMO condition for some $\delta, R_0>0$  if
\begin{equation*}
[A]_{R_0}:=\mathop {\sup }\limits_{y\in \mathbb{R}^n,0<r\leq R_0}\fint_{B_r(y)}\Theta(A,B_r(y))(x)dx \leq \delta,
\end{equation*}         
where 
\begin{equation*}
\Theta(A,B_r(y))(x):=\mathop {\sup }\limits_{\xi\in\mathbb{R}^n\backslash\{0\}}\frac{|A(x,\xi)-\overline{A}_{B_r(y)}(\xi)|}{|\xi|^{p-1}},
\end{equation*}
and $\overline{A}_{B_r(y)}(\xi)$ denotes the average of $A(\cdot,\xi)$ over the ball $B_r(y)$, i.e.,
\begin{equation*}
\overline{A}_{B_r(y)}(\xi):=\fint_{B_r(y)}A(x,\xi)dx=\frac{1}{|B_r(y)|}\int_{B_r(y)}A(x,\xi)dx.
\end{equation*}                                                              
\end{definition}

As far as the regularity of the boundary of $\Om$ is concerned,   we  require that it be sufficiently flat in the sense of Reifenberg  \cite{55Re}.
 Namely, at each
boundary point and every scale, we ask that the boundary of $\Omega$  be trapped between two hyperplanes separated by a distance that
depends on the scale. This class of domains includes $C^1$ domains and  Lipschitz domains with sufficiently small Lipschitz constants (see \cite{Tor}).  Moreover, they also include certain domains with fractal boundaries and thus allow for a wide range of potential applications.    
\begin{definition}
 Given $\delta\in (0, 1)$ and $R_0>0$, we say that $\Omega$ is a $(\delta, R_0)$-Reifenberg flat domain if for every $x\in \partial \Omega$
 and every $r\in (0, R_0]$, there exists a
 system of coordinates $\{ z_{1}, z_{2}, \dots,z_{n}\}$,
 which may depend on $r$ and $x$, so that  in this coordinate system $x=0$ and that
\[
B_{r}(0)\cap \{z_{n}> \delta r \} \subset B_{r}(0)\cap \Omega \subset B_{r}(0)\cap \{z_{n} > -\delta r\}.
\]
\end{definition}

  In this paper, all solutions of   \eqref{5hh070120148} and \eqref{Riccati} with a finite signed measure $\mu$ in $\Omega$ will be understood in the     renormalized sense (see \cite{11DMOP}). For $\mu\in\mathfrak{M}_b(\Omega)$ (the set of finite signed measures in $\Omega$), we will tacitly extend it by zero to $\Omega^c:=\mathbb{R}^n\setminus\Omega$.	We let $\mu^+$ and $\mu^-$  be the positive and negative parts, respectively, of a measure $\mu\in\mathfrak{M}_b(\Omega)$. We denote by $\mathfrak{M}_0(\Omega)$ the space of finite signed measures in $\Omega$ which are absolutely continuous with respect to the capacity $c^{\Omega}_{1,p}$. Here  $c^{\Omega}_{1,p}$ is the $p$-capacity defined  for each compact set $K\subset\Omega$ by
  \begin{equation*}
  c^{\Omega}_{1,p}(K)=\inf\left\{\int_{\Omega}{}|{\nabla \varphi}|^pdx:\varphi\geq \chi_K,\varphi\in C^\infty_c(\Omega)\right\},
  \end{equation*}
	where $\chi_{K}$ is the characteristic function of the set $K$.
	  We also denote by $\mathfrak{M}_s(\Omega)$ the space of finite signed measures in $\Omega$ concentrated on a set of zero $c^{\Omega}_{1,p}$-capacity. It is known that  any $\mu\in\mathfrak{M}_b(\Omega)$ can be written  uniquely  in the form $\mu=\mu_0+\mu_s$ where $\mu_0\in \mathfrak{M}_0(\Omega)$ and $\mu_s\in \mathfrak{M}_s(\Omega)$ (see \cite{FS}).
  It is also known  that any  $\mu_0\in \mathfrak{M}_0(\Omega)$ can be written in the form $\mu_0=f-{\rm div}( F)$ where $f\in L^1(\Omega)$ and $F\in L^{\frac{p}{p-1}}(\Omega,\mathbb{R}^n)$.
  
  For $k>0$, we define the usual  two-sided truncation operator $T_k$ by
	$$T_k(s)=\max\{\min\{s,k\},-k\}, \qquad s\in\mathbb{R}.$$ 
	
	For our purpose, the following notion of gradient is needed.
	If $u$ is a measurable function defined  in $\Omega$, finite a.e., such that $T_k(u)\in W^{1,p}_{loc}(\Omega)$ for any $k>0$, then there exists a measurable function $v:\Omega\to \mathbb{R}^n$ such that $\nabla T_k(u)= v \chi_{\{|u|< k\}} $ 
  a.e. in $\Omega$  for all $k>0$ (see \cite[Lemma 2.1]{bebo}). In this case, we define the gradient $\nabla u$ of $u$ by $\nabla u:=v$. It is known that  $v\in L^1_{loc}(\Omega, \mathbb{R}^n)$ if and only if  $u\in W^{1,1}_{loc}(\Omega)$ and then $v$ is the usual weak gradient of $u$. On the other hand, for $1<p\leq 2-\frac{1}{n}$, by looking at the fundamental solution we see that in general distributional solutions of \eqref{5hh070120148}  may not even belong to $u\in W^{1,1}_{loc}(\Omega)$.\\

	The notion of renormalized solutions is a generalization of that of entropy solutions introduced in \cite{bebo} and \cite{BGO}, where the right-hand side is assumed to be  in $L^1(\Omega)$ or in  $\mathfrak{M}_{0}(\Omega)$.
Several equivalent definitions of renormalized solutions
were  given  in \cite{11DMOP}. Here we use the following one:

  
  \begin{definition} \label{derenormalized} Let $\mu=\mu_0+\mu_s\in\mathfrak{M}_b(\Omega)$, with $\mu_0\in \mathfrak{M}_0(\Omega)$ and $\mu_s\in \mathfrak{M}_s(\Omega)$. A measurable  function $u$ defined in $\Omega$ and finite a.e. is called a renormalized solution of \eqref{5hh070120148}
  	if $T_k(u)\in W^{1,p}_0(\Omega)$ for any $k>0$, $|{\nabla u}|^{p-1}\in L^r(\Omega)$ for any $0<r<\frac{n}{n-1}$, and $u$ has the following additional property. For any $k>0$ there exist  nonnegative Radon measures $\lambda_k^+, \lambda_k^- \in\mathfrak{M}_0(\Omega)$ concentrated on the sets $\{u=k\}$ and $\{u=-k\}$, respectively, such that 
  	$\mu_k^+\rightarrow\mu_s^+$, $\mu_k^-\rightarrow\mu_s^-$ in the narrow topology of measures and  that
  	\[
  	\int_{\{|u|<k\}}\langle A(x,\nabla u),\nabla \varphi\rangle
  	dx=\int_{\{|u|<k\}}{\varphi d}{\mu_{0}}+\int_{\Omega}\varphi d\lambda_{k}%
  	^{+}-\int_{\Omega}\varphi d\lambda_{k}^{-},
  	\]
  	for every $\varphi\in W^{1,p}_0(\Omega)\cap L^{\infty}(\Omega)$.
  	  \end{definition}

Here we recall that a sequence  $\{\mu_{k}\} \subset
\mathfrak{M}_{b}(\Omega)$ is said to converge in the narrow topology of measures to $\mu \in
\mathfrak{M}_{b}(\Omega)$ if
$\lim_{k\rightarrow\infty}\int_{\Omega}\varphi \, d\mu_{k}=\int_{\Omega}\varphi \,
d\mu,$
for every bounded and continuous function $\varphi$ on $\Omega$.

It is known that if $\mu\in \mathfrak{M}_0(\Omega)$ then there is one and only one  renormalized solution of 
\eqref{5hh070120148} (see \cite{BGO, 11DMOP}). However, to the best of our knowledge, for a general $\mu\in \mathfrak{M}_b(\Omega)$ the uniqueness of renormalized solutions of \eqref{5hh070120148} is still an open problem.

Recall that the  Hardy-Littlewood maximal function ${\bf M}$ is defined for each locally integrable function  $f$ in $\mathbb{R}^{n}$ by
\begin{equation*}
{\bf M}(f)(x)=\sup_{\rho>0}\fint_{B_\rho(x)}|f(y)|dy \qquad \qquad \forall x\in\mathbb{R}^{n}.
\end{equation*}

For a signed measure $\mu$ in $\RR^n$,   the first order fractional maximal function of $\mu$,  $\mathbf{M}_1(\mu)$, is defined by
\begin{align*}
\mathbf{M}_1(\mu)(x):=\sup_{\rho>0}\frac{|\mu|(B_\rho(x))}{\rho^{n-1}} \qquad \qquad \forall  x\in \mathbb{R}^n.
\end{align*}

A  nonnegative  function $w\in L^1_{loc}(\mathbb{R}^{n})$ is said to be an $\mathbf{A}_{\infty}$ weight if there are two positive constants $C$ and $\nu$ such that
$$w(E)\le C \left(\frac{|E|}{|B|}\right)^\nu w(B),
$$
for all balls $B\subset\RR^n$ and all measurable subsets $E\subset B$. The pair $(C,\nu) $ is called the $\mathbf{A}_\infty$ constants of $w$ and is denoted by $[w]_{\mathbf{A}_\infty}$.  It is well-known that 
$$\mathbf{A}_\infty=\bigcup_{q>1} \mathbf{A}_q,$$
where we say that a nonnegative  function $w\in L^1_{loc}(\mathbb{R}^{n})$ belongs to the Muckenhoupt $\mathbf{A}_q$ class, $q>1$, if 
$$[w]_{\mathbf{A}_q}:= \sup_{{\rm balls~} B\subset\RR^n} \left(\fint_{B} wdx\right) \left(\fint_{B} w^{\frac{1}{1-q}} dx \right)^{q-1} <+\infty.$$

Our first result concerns with a weighted `good-$\lambda$' type inequality  for renormalized solutions of 			\eqref{5hh070120148}.
																				
\begin{theorem}\label{5hh23101312}   Let $w\in {\bf A}_\infty$, $\mu\in\mathfrak{M}_b(\Omega)$, $1<p\leq \frac{3n-2}{2n-1}$, and $\gamma_1\in \left(0, \frac{(p-1)n}{n-1}\right)$.   For any $\varepsilon>0,R_0>0$, one can find  constants $\delta_1=\delta_1(n,p,\Lambda,\gamma_1,\varepsilon,[w]_{{\bf A}_\infty})\in (0,1)$, $\delta_2=\delta_2(n,p,\Lambda, \gamma_1,\varepsilon,[w]_{{\bf A}_\infty},diam(\Omega)/R_0)\in (0,1)$, and $\Lambda_0=\Lambda_0(n,p,\Lambda, \gamma_1)>1$ such that if $\Omega$ is $(\delta_1,R_0)$-Reifenberg flat  and $[A]_{R_0}\le \delta_1$ then  for any renormalized solution $u$ to \eqref{5hh070120148} with $|\nabla u|\in L^{2-p}(\Om)$, we have 
\begin{align}\label{101120144}
&w(\{( {\bf M}(|\nabla u|^{\gamma_1}))^{1/\gamma_1}>\Lambda_0\lambda, (\mathbf{M}_1(\mu))^{\frac{1}{p-1}}\le \delta_2\lambda \}\cap U_{\epsilon,\lambda}\cap \Omega) \nonumber\\
&\qquad\leq C\varepsilon w(\{ ({\bf M}(|\nabla u|^{\gamma_1}))^{1/\gamma_1}> \lambda\}\cap \Omega), 
\end{align}
for any $\lambda>0$. Here  $U_{\epsilon,\lambda}=\{{\bf M}(|\nabla u|^{2-p}))^{\frac{1}{2-p}}\leq \varepsilon^{-1}\lambda\}$ and the constant $C$  depends only on $n,p,\Lambda,diam(\Omega)/R_0$, and $[w]_{{\bf A}_\infty}$. 
\end{theorem}

The presence of the set 	$U_{\epsilon,\lambda}$ in \eqref{101120144} makes Theorem 	\ref{5hh23101312} different from \cite[Theorem 1.5]{QH4} in which the case $\frac{3n-2}{2n-1}<p\leq 2- \frac{1}{n}$ was treated.  However, Theorem 	\ref{5hh23101312} can be used to obtain the following existence and regularity of solutions to
\eqref{5hh070120148}, which extends the results of \cite{QH4, 55Ph2} to the case $1<p\leq \frac{3n-2}{2n-1}$.

 \begin{theorem} \label{101120143} Let $\mu\in \mathfrak{M}_b(\Omega)$ and $1<p\leq \frac{3n-2}{2n-1}$.   For any $2-p< q<\infty$ and $w\in \mathbf{A}_{\frac{q}{2-p}}$, we can find  $\delta=\delta(n,p,\Lambda, q, [w]_{\mathbf{A}_{\frac{q}{2-p}}})\in (0,1)$ such that if $\Omega$ is  $(\delta,R_0)$-Reifenberg flat   and $[A]_{R_0}\le \delta$ for some $R_0>0$, then  there exists a renormalized solution $u$ to \eqref{5hh070120148} such that                           
                  \begin{equation}\label{mainbound4}
                              \|\nabla u\|_{L^{q}_w(\Omega)}\leq C \|[\mathbf{M}_1(\mu)]^{\frac{1}{p-1}}\|_{L^{q}_w(\Omega)}.
                                       \end{equation} 
                                        Here the constant $C$ depends only  on $n,p,\Lambda,q, [w]_{ \mathbf{A}_{\frac{q}{2-p}}}$, and $diam(\Omega)/R_0$.               
\end{theorem}

\begin{remark}\label{uniq} By uniqueness of renormalized solutions with data in $\mathfrak{M}_{0}(\Om)$, we see that  \eqref{mainbound4} indeed holds for any renormalized solution $u$ to 
\eqref{5hh070120148} with datum $\mu\in\mathfrak{M}_{0}(\Om)$.	
\end{remark}

\begin{remark} Theorem  \ref{101120143}  also holds if we replace the weighted Lebesgue space  $L^{q}_w(\Omega)$ with the more general weighted Lorentz space 
	$L^{q,s}_w(\Omega)$, $0<s\leq \infty$, (see, e.g., \cite{QH4, 55Ph2}).
\end{remark}

We now describe our results in regard to the Riccati type equation  \eqref{Riccati}.  For this, we shall need the notion of capacity
associated to the Sobolev space $W^{1, s}(\RR^n)$,  $1<s<+\infty$.
For a compact set $K\subset\RR^n$,  we  define
\begin{equation*}
{\rm Cap}_{1,  s}(K)=\inf\Big\{\int_{\RR^n}(|\nabla \varphi|^s +\varphi^s) dx: \varphi\in C^\infty_0(\RR^n),
\varphi\geq \chi_K \Big\}.
\end{equation*}
%

As \cite[Theorem 1.9]{QH4} and \cite[Theorem 1.6]{55Ph2}, we  obtain the following sharp existence result but now for the case  $1<p\leq \frac{3n-2}{2n-1}$.
\begin{theorem}\label{main-Ric}
Let $1<p\leq \frac{3n-2}{2n-1}$ and  $q\geq 1$. 
There exists a constant  $\delta=\delta(n, p, \Lambda, q)\in (0, 1)$ such that the following holds. 
Suppose that  $[A]_{R_0}\leq \delta$ and $\Om$ is $(\delta, R_0)$-Reifenberg flat for some $R_0>0$. Then there exists 
$c_0=c_0(n, p, \Lambda, q, diam(\Om), R_0)>0$ such that if $\mu$ is a finite signed measure in $\Om$ with 
\begin{equation}\label{capcondi} 
|\mu|(K) \leq c_0\, {\rm Cap}_{1,\, \frac{q}{q-p+1}}(K)
\end{equation}
for all compact sets $K\subset\Om$, then there exists a renormalized solution $u\in W_0^{1, q}(\Om)$ to the Riccati type equation \eqref{Riccati} such that 
$$
\int_{K} |\nabla u|^q  \leq C\, {\rm Cap}_{1,\, \frac{q}{q-p+1}}(K)
$$
for all compact sets $K\subset\Om$. Here the constant $C$ depends only on $n, p, \Lambda, q,  diam(\Om)$, and $R_0$.
\end{theorem}

There is a vast literature on equations with a power growth in the gradient of the form \eqref{Riccati}. We only mention here the pioneering work \cite{HMV} which originally used 
capacity to treat   \eqref{Riccati} in the `linear' case $p=2$. For other contributions, see, e.g.,  the references in \cite{QH4}.

It is known  that  condition \eqref{capcondi} is sharp in the sense that   if \eqref{Riccati} has a solution with $\mu$ being  nonnegative and compactly supported in $\Om$ then  \eqref{capcondi} holds with  a different constant $c_0$ (see \cite{HMV, Ph1}). It is more general than  
  the Marcinkiewicz space condition $\mu\in L^{\frac{n(q-p+1)}{q},  \infty}(\Om)$, $q>\frac{n(p-1)}{n-1}$, (with a small norm), or the Fefferman-Phong type condition involving Morrey spaces (see, e.g., \cite{Ph1}). Moreover, Theorem \ref{main-Ric} implies that any 
compact set $K\subset\Om$ that is removable for the equation $-{\rm div} (A(x, \nabla u)) =|\nabla u|^q$ must be small in the sense that ${\rm Cap}_{1,\frac{q}{q-p+1}}(K)=0$ (see \cite[Theorem 3.9]{Ph1}).


The paper is organized as follows. In Section \ref{sec-2} we obtain some important comparison estimates that are needed for the proof of Theorem \ref{5hh23101312}.
The proofs  Theorems \ref{5hh23101312}, \ref{101120143}, and \ref{main-Ric}  are given in Sections  \ref{sec-3}, \ref{sec-4}, and \ref{sec-5}, respectively.  

    \section{Comparison estimates}\label{sec-2}
Let $u\in W_{loc}^{1,p}(\Omega)$ be a solution of \eqref{5hh070120148}. For each ball $B_{2R}=B_{2R}(x_0)\subset\subset\Omega$, we consider the unique solution $w\in W_{0}^{1,p}(B_{2R})+u$
    to the  equation 
    \begin{equation}
    \label{111120146}\left\{ \begin{array}{rcl}
    - \operatorname{div}\left( {A(x,\nabla w)} \right) &=& 0 \quad in \quad B_{2R}, \\ 
    w &=& u\quad \text{on} \quad \partial B_{2R}.  
    \end{array} \right.
    \end{equation}

   Then  we have the following estimate for the difference $\nabla u-\nabla w$  in terms of the total variation of $\mu$ in $B_{2R}$ and the norm of $\nabla u$ in $L^{2-p}(B_{2R})$.\
	This estimate holds true for all $1<p\leq \frac{3n-2}{2n-1}$. For earlier results of this type for $p>\frac{3n-2}{2n-1}$, we refer to \cite{Mi2,55DuzaMing,Duzamin2,QH4}. 
		
    \begin{lemma}\label{1111201410}Let $u$ and $w$ be as in \eqref{111120146}, and assume that   $1<p\leq \frac{3n-2}{2n-1}$. Then
    	\begin{align*}
    	\left(\fint_{B_{2R}}|\nabla (u-w)|^{\gamma_1}\right)^{\frac{1}{\gamma_1}}\leq C\left(\frac{|\mu|(B_{2R})}{R^{n-1}}\right)^{\frac{1}{p-1}}+\frac{|\mu|(B_{2R})}{R^{n-1}}\fint_{B_{2R}}|\nabla u|^{2-p},
    	\end{align*}
    	for any $0<\gamma_1<\frac{n(p-1)}{n-1}$.
    \end{lemma}
\begin{proof} By scaling invariance, we may assume that  $ |\mu|(B_{2R})=1$ and $B_{2R}=B_{2}$. For $k>0$, 
	using $\varphi= T_{2k}(u-w)$ as a test function, we have 
$$
	\int_{B_{2}\cap\{|u-w|<2k\}}g(u,w)dx\leq Ck , \text{~~with~~} g(u,w)=\frac{|\nabla (u-w)|^2}{(|\nabla w|+|\nabla u|)^{2-p}}.$$
	Set $E_k=B_{2}\cap\{k<|u-w|<2k\}$, and $F_k=B_{2}\cap\{|u-w|>k\}$. Using  H\"older's inequality yields 
	\begin{align*}
	&k\left|\left\{x:|u-w|>2k\right\}\cap B_{2}\right|^{\frac{n-1}{n}}\leq C \left(\int_{B_{ 2}}|T_{2k}(u-w)-T_{k}(u-w)|^{\frac{n}{n-1}}\right)^{\frac{n-1}{n}}\\&\leq  C \int_{E_k}|\nabla (u-w)|\\&\leq C  \int_{E_k} g(u,w)^{1/p}+g(u,w)^{1/2}|\nabla u|^{\frac{2-p}{2}}\\&\leq C|E_k|^{\frac{p-1}{p}} \left(\int_{E_k} g(u,w)\right)^{1/p}+\left(\int_{E_k} g(u,w)\right)^{1/2} \left(\int_{E_k} |\nabla u|^{2-p}\right)^{\frac{1}{2}}.
	\end{align*}
Here in the third inequality we  used that  
\begin{equation}
\label{z1}	|\nabla (u-w)|\leq C \left( g(u,w)^{1/p}+g(u,w)^{1/2}|\nabla u|^{\frac{2-p}{2}}\right),
\end{equation}
which holds provided $1<p<2$.
	As 
	$
	\int_{E_k} g(u,w)\leq Ck,$ we thus find
	\begin{align*}
	&k^{1/2}\left|F_{2k}\right|^{\frac{n-1}{n}}
	\leq Ck^{-1/2+1/p}|F_k|^{\frac{p-1}{p}}+CQ_1^{\frac{2-p}{2}},
	\end{align*}
	where we set  $Q_1=||\nabla u||_{L^{2-p}(B_{2})}$.

Note that we can write for $\epsilon\geq 0$,
\begin{align*}
&k^{1/2}\left|F_{2k}\right|^{\frac{n-1}{n}+\epsilon}
\leq Ck^{-1/2+1/p}|F_k|^{\frac{p-1}{p}+\epsilon}+CQ_1^{\frac{2-p}{2}},
\end{align*}
which implies 
\begin{equation*}
||u-w||_{L^{\frac{1}{2\left(\frac{n-1}{n}+\epsilon\right)},\infty}(B_2)}^{1/2}\leq C ||u-w||_{L^{\frac{1/p-1/2}{\frac{p-1}{p}+\epsilon},\infty}(B_2)}^{1/p-1/2}+CQ_1^{\frac{2-p}{2}}.
\end{equation*}

Choosing 
$$
\epsilon=\frac{3n-2-p(2n-1)}{2(p-1)n},
$$
we have 
$$
\frac{1}{2\left(\frac{n-1}{n}+\epsilon\right)}=\frac{1/p-1/2}{\frac{p-1}{p}+\epsilon}=\frac{n(p-1)}{n-p}.
$$

Thus, using Holder's inequality we obtain
\begin{equation}\label{z2}
||u-w||_{L^{\frac{n(p-1)}{n-p},\infty}(B_2)}\leq C+CQ_1^{2-p}.
\end{equation}

	For $k,\lambda\geq 0$, and $q=\frac{n(p-1)}{n-p}$,  we have 
	\begin{align*}
|\{x: g(u,w)>&\lambda\}\cap B_{2}|\leq|\{x:|u-w|>k\}\cap B_{2}|+
	\\&\quad  +\frac{1}{\lambda}\int_{0}^{\lambda}|\{x:|u-w|\leq k,g(u,w)>s\}\cap B_{2}| ds
	\\&\leq Ck^{-q}||u-w||_{L^{q,\infty}(B_2)}^q+\frac{1}{\lambda}\int_{B_{2}\cap\{x:|u-w|\leq k\}}g(u,w) dx
	\\&\leq Ck^{-q}||u-w||_{L^{q,\infty}(B_2)}^q+\frac{C k}{\lambda}.
	\end{align*}

	 Then choosing $$k=\left[\lambda ||u-w||_{L^{q,\infty}(B_2)}^q\right]^{\frac{1}{1+q}},$$
	we obtain
$$
	\lambda^{\frac{q}{1+q}}|\{x:g(u,w)>\lambda\}\cap B_{2}|\leq C ||u-w||_{L^{q,\infty}(B_2)}^{\frac{q}{q+1}},$$
	for all $\lambda>0$. This means 
$$
	||g(u,w)||_{L^{\frac{q}{q+1},\infty}}\leq  ||u-w||_{L^{q,\infty}(B_2)}.
$$
	
	Let $\gamma_1\in(0,\frac{n(p-1)}{n-1})$.
	By \eqref{z1} and Holder's inequality with exponents $\frac{2}{p}$ and $\frac{2}{2-p}$:
	\begin{align*}
	\int_{B_2}|\nabla (u-w)|^{\gamma_1}&\leq C\int_{B_2} \left[g(u,w)^{\gamma_1/p}+g(u,w)^{\gamma_1/2}|\nabla u|^{\frac{\gamma_1(2-p)}{2}}\right]
	\\ &\leq C||g(u,w)||_{L^{\frac{q}{q+1},\infty}}^{\gamma_1/p} +C\left(\int_{B_2} g(u,w)^{\gamma_1/p}\right)^{\frac{p}{2}}\left(\int_{B_2}|\nabla u|^{\gamma_1}\right)^{\frac{2-p}{2}}
	\\ &\leq C||g(u,w)||_{L^{\frac{q}{q+1},\infty}}^{\gamma_1/p} +C||g(u,w)||_{L^{\frac{q}{q+1},\infty}}^{\gamma_1/2}Q_1^{\frac{(2-p)\gamma_1}{2}}.
	\end{align*}
Here we  used the fact that 
	$$\frac{\gamma_1}{p}<\frac{q}{q+1}, \quad \gamma_1\leq 2-p.$$
	Combining this with \eqref{z2} yields
$$
	\int_{B_2}|\nabla (u-w)|^{\gamma_1}\leq C+CQ_1^{(2-p)\gamma_1}
$$
which implies the result. 
\end{proof}

\medskip
Just as \cite[Proposition 2.3]{QH4}, we also obtain from Lemma \ref{1111201410} the following result.

\begin{proposition} \label{inter} Let $0<\gamma_1<\frac{(p-1)n}{n-1}$. There exists $v\in W^{1,p}(B_R)\cap W^{1,\infty}(B_{R/2})$ such that for any $\varepsilon>0$, 
$$
	\|\nabla v\|_{L^\infty(B_{R/2})}\leq C_\varepsilon \left[\frac{|\mu|(B_{2R})}{R^{n-1}}\right]^{\frac{1}{p-1}}+C \left(\fint_{B_{2R}}|\nabla u|^{\gamma_1}\right)^{\frac{1}{\gamma_1}}+\varepsilon \left(\fint_{B_{2R}}|\nabla u|^{2-p}\right)^{\frac{1}{2-p}},
$$
	and
	\begin{align*}
	&\left(\fint_{B_{R}}|\nabla u-\nabla v|^{\gamma_1}dx\right)^{\frac{1}{\gamma_1}} \leq C_\varepsilon \left[\frac{|\mu|(B_{2R})}{R^{n-1}}\right]^{\frac{1}{p-1}}+\\
	&\qquad\qquad  +C(([A]_{R_0})^{\kappa} +\varepsilon)\left(\fint_{B_{2R}}|\nabla u|^{\gamma_1}\right)^{\frac{1}{\gamma_1}}+\varepsilon \left(\fint_{B_{2R}}|\nabla u|^{2-p}\right)^{\frac{1}{2-p}},
	\end{align*}
for some $C_\varepsilon=C(n,p,\Lambda,\varepsilon)>0$. Here $\kappa$ is a constant in $(0,1)$.
\end{proposition}
\medskip

      Lemmas  \ref{1111201410} and Proposition \ref{inter} can be extended  up to  boundary. We recall that $\Omega$ is $(\delta_0,R_0)$-Reifenberg flat  with $\delta_0<1/2$. Fix $x_0\in \partial\Omega$ and $0<R<R_0/10$. With $u\in W_0^{1,p}(\Omega)$ being a solution to \eqref{5hh070120148}, 
      we now consider the unique solution $w\in W_{0}^{1,p}(\Omega_{10R}(x_0))+u$
       to the following equation 
      \begin{equation}
      \label{111120146*}\left\{ \begin{array}{rcl}
      - \operatorname{div}\left( {A(x,\nabla w)} \right) &=& 0 \quad ~~~\text{in}\quad \Omega_{10R}(x_0), \\ 
      w &=& u\quad \quad \text{on} \quad \partial \Omega_{10R}(x_0),
      \end{array} \right.
      \end{equation}
      where we define   $\Omega_{10R}(x_0)=\Omega\cap B_{10R}(x_0)$. 
			
  Then we have the following analogue of Lemmas  \ref{1111201410}.
  \begin{lemma}\label{111120149"} Let $0<\gamma_1<\frac{(p-1)n}{n-1}$, $1<p\leq \frac{3n-2}{2n-1}$, and  let $u, w$ be as in \eqref{111120146*}. Then we  have
  	\begin{align*}\nonumber
  \left(	\fint_{B_{10R}(x_0)}|\nabla (u-w)|^{\gamma_1}dx\right)^{\frac{1}{\gamma_1}}&\leq C\left[\frac{|\mu|(B_{10R}(x_0))}{R^{n-1}}\right]^{\frac{1}{p-1}}+\\
&\qquad +	C \frac{|\mu|(B_{10R}(x_0))}{R^{n-1}}\fint_{B_{10R}(x_0)}|\nabla u|^{2-p}dx.
  	\end{align*}
  \end{lemma}

Using Lemma \ref{111120149"} and \cite[Corollary 2.13]{55Ph2} we can derive the following boundary version of Proposition \ref{inter}.
   \begin{proposition} \label{boundary} Let $0<\gamma_1<\frac{(p-1)n}{n-1}$. For any $\varepsilon>0$, there exists a constant  $\delta_0=\delta_0(n,p,\Lambda,\varepsilon)\in(0,1)$ such that the following holds. If $\Omega$ is $(\delta_0,R_0)$-Reifenberg flat  and $u\in W_{0}^{1,p}(\Omega)$, $x_0\in\partial \Omega$, and $0<R<R_0/10$, then there exists a function $V\in W^{1,\infty}(B_{R/10}(x_0))$ 
such that   	
\begin{align*}
   	\|\nabla V\|_{L^\infty(B_{R/10})}&\leq C_\varepsilon \left[\frac{|\mu|(B_{10R})}{R^{n-1}}\right]^{\frac{1}{p-1}}+C \left(\fint_{B_{10R}}|\nabla u|^{\gamma_1}\right)^{\frac{1}{\gamma_1}}\\
   	&\qquad \qquad  + \varepsilon\left(\fint_{B_{10R}}|\nabla u|^{2-p}\right)^{\frac{1}{2-p}},
   	\end{align*}
   	and
   	\begin{align*}
   &\left(\fint_{B_{R/10}}|\nabla (u-V)|^{\gamma_1}dx\right)^{\frac{1}{\gamma_1}}\leq C_\varepsilon \left[\frac{|\mu|(B_{10R})}{R^{n-1}}\right]^{\frac{1}{p-1}}\\&~~~~~+C(([A]_{R_0})^{\kappa} +\varepsilon)\left(\fint_{B_{10R}}|\nabla u|^{\gamma_1}\right)^{\frac{1}{\gamma_1}}+ \varepsilon\left(\fint_{B_{10R}}|\nabla u|^{2-p}\right)^{\frac{1}{2-p}},
   	\end{align*}
   	for some $C_\varepsilon=C(n,p,\Lambda,\varepsilon)>0$. Here $\kappa$ is a constant in $(0,1)$ and  the balls are centered at $x_0$.
   \end{proposition} 

\section{Proof of Theorem \ref{5hh23101312}}\label{sec-3}

The proof of  Theorem  \ref{5hh23101312} is based on Propositions \ref{inter} and \ref{boundary} and the following technical lemma  (see \cite{55MePh2}).

\begin{lemma}\label{5hhvitali2} Let $\Omega$ be a $(\delta,R_0)$-Reifenberg flat domain with $\delta<1/4$ and let $w$ be an $\mathbf{A}_\infty$ weight. Suppose that the sequence of balls $\{B_r(y_i)\}_{i=1}^L$ with centers $y_i\in\overline{\Omega}$ and  radius $r\leq R_0/4$ covers $\Omega$.  Let $E\subset F\subset \Omega$ be measurable sets for which 
 there exists $0<\varepsilon<1$ such that  
\begin{description}
 \item[1.] $w(E)<\varepsilon w(B_r(y_i))$ for all $i=1,...,L$, and 
 \item[2.] for all $x\in \Omega$, $\rho\in (0,2r]$, we have       
	 $w(E\cap B_\rho(x))\geq \varepsilon w(B_\rho(x)) \Longrightarrow B_\rho(x)\cap \Omega\subset F$. 
	\end{description}
 Then $
	w(E)\leq C \varepsilon w(F)$         
	for a constant $C$ depending only on $n$ and $[w]_{\mathbf{A}_\infty}$.
\end{lemma}

\medskip

\noindent \begin{proof}[Proof of Theorem \ref{5hh23101312}] The proof  is reminiscent of that of \cite[Theorem 1.5]{QH4} (see also \cite[Theorem 1.4]{55Ph2}, \cite[Theorem 8.4]{55QH2}, and \cite[Theorem 3.1]{55QH3}).

Let $R=diam(\Omega)$. Suppose that $0<\gamma_1<\frac{n(p-1)}{n-1}$ and  $u$ is a renormalized solution of \eqref{5hh070120148}  such that $|\nabla u|\in L^{2-p}(\Om)$. By \cite[Theorem 4.1]{11DMOP} we have 
$$
	\|\nabla u\|_{L^{\frac{(p-1)n}{n-1},\infty}(\Omega)}\leq C\left[|\mu|(\Omega)\right]^{\frac{1}{p-1}},
$$
	which implies  that
$$
\left(	\frac{1}{R^n}\int_{\Omega}|\nabla u|^{\gamma_1}\right)^{1/\gamma_1}\leq C_{\gamma_1} \left[\frac{|\mu|(\Omega)}{R^{n-1}}\right]^{\frac{1}{p-1}}.
$$

For $k>0$, let $\mu_0,\lambda_k^+,\lambda_k^-$ be as in Definition \ref{derenormalized}. Let $u_k\in W_0^{1,p}(\Omega)$ be the unique solution of the equation
	\begin{equation*}
	\left\{
	\begin{array}[c]{rcl}
	-\text{div}(A(x,\nabla u_k))&=&\mu_{k} \quad \text{in } \quad\Omega,\\
	{u}_{k}&=&0\quad\text{on } \quad \partial\Omega,\\
	\end{array}
	\right.  
	\end{equation*}
	where we set $\mu_k=\chi_{\{|u|<k\}}\mu_0+\lambda_k^+-\lambda_k^-$. 
	 
Note that  we have  $u_k=T_k(u)$ and $\mu_k\rightarrow \mu$ 
in the narrow topology of measures (see \cite[Remark 2.32]{11DMOP}). Thus,  
	\begin{equation}\label{grad-appr}
	\nabla u_k\to \nabla u  \quad \text{in} \quad L^{\gamma_1}(\Omega)\cap L^{2-p}(\Om).
	\end{equation}
	
Let us set 
$$F_\lambda=\{ ({\bf M}(|\nabla u|^{\gamma_1}))^{1/\gamma_1}> \lambda\}\cap \Omega,$$
and
$$E_{\lambda,\delta_2}=\left\{({\bf M}(|\nabla u|^{\gamma_1}))^{1/\gamma_1}>\Lambda_0\lambda,  (\mathbf{M}_1(\mu))^{\frac{1}{p-1}}\le \delta_2\lambda \right\}\cap U_{\epsilon,\lambda}\cap \Omega,$$ 
where 
$$U_{\epsilon,\lambda}=\{({\bf M}(|\nabla u|^{2-p}))^{\frac{1}{2-p}}\leq \varepsilon^{-1}\lambda \},$$
and $\delta_2\in (0,1)$, $\lambda>0$. The constant  $\Lambda_0$  depends only on $n,p,\gamma_1,\Lambda$ and is to be  chosen  later.  
	
	Also, let $\{y_i\}_{i=1}^L\subset \Omega$ and a ball $B_0$ with radius $2R$ such that 
	 $$
	 \Omega\subset \bigcup\limits_{i = 1}^L {{B_{r_0}}({y_i})}  \subset {B_0},\qquad  {\rm where}\  r_0=\min\{R_0/1000,R\}.$$

	As in the proof of  \cite[Theorem 1.5]{QH4}, we have 
	 \begin{equation}\label{5hh2310131}
	 w(E_{\lambda,\delta_2})\leq \varepsilon w({B_{r_0}}({y_i})) ~~\forall \lambda>0, \forall i=1,2,\dots,L,
	 \end{equation}
	 provided $\delta_2=\delta_2(n,p,\Lambda,\epsilon,[w]_{\mathbf{A}_\infty},R/R_0)>0$ is small enough.

	In order to apply Lemma \ref{5hhvitali2} we now verify that for all $x\in \Omega$, $r\in (0,2r_0]$, and $\lambda>0$ we have
	 \begin{equation}\label{2nd-check}
	 w(E_{\lambda,\delta_2}\cap B_r(x))\geq \varepsilon w(B_r(x)) \Longrightarrow B_r(x)\cap \Omega\subset F_\lambda,
	 \end{equation}
	 provided $\delta_2$ is small enough depending on $n,p,\Lambda, \gamma_0,\epsilon,[w]_{\mathbf{A}_\infty},R/R_0$ .

	Indeed,  take $x\in \Omega$ and $0<r\leq 2r_0$.
	 By contraposition, assume that $B_r(x)\cap \Omega\cap F^c_\lambda\not= \emptyset$ and $E_{\lambda,\delta_2}\cap B_r(x)\not = \emptyset$, i.e., there exist $x_1,x_2\in B_r(x)\cap \Omega$ such that 
	 \begin{equation}\label{Mx1}
	 \left[{\bf M}(|\nabla u|^{\gamma_1})(x_1)\right]^{1/\gamma_1}\leq \lambda,
	 \end{equation}
	  and 
	  \begin{equation}\label{MM1x2}
	  {\bf M}\left(|\nabla u|^{2-p}\right)(x_2)\leq (\varepsilon^{-1}\lambda)^{2-p}, \quad  \mathbf{M}_1(\mu)(x_2)\le (\delta_2 \lambda)^{p-1}.
	 \end{equation}

	 We need to prove that
	 \begin{equation}\label{5hh2310133}
	 w(E_{\lambda,\delta_2}\cap B_r(x))< \varepsilon w(B_r(x)). 
	 \end{equation}

	 It follows from \eqref{Mx1} that
$$
	 {\bf M}(|\nabla u|^{\gamma_1})(y)^{\frac{1}{\gamma_1}}\leq \max\{\left[{\bf M}\left(\chi_{B_{2r}(x)}|\nabla u|^{\gamma_1}\right)(y)\right]^{\frac{1}{\gamma_1}},3^{n}\lambda\}~~\forall y\in B_r(x).
$$

 Therefore, for all $\lambda>0$ and $\Lambda_0\geq 3^{n}$, we find
\begin{align}
&E_{\lambda,\delta_2}\cap B_r(x)\nonumber\\
&=\left\{{\bf M}\left(\chi_{B_{2r}(x)}|\nabla u|^{\gamma_1}\right)^{\frac{1}{\gamma_1}}>\Lambda_0\lambda, (\mathbf{M}_{1}(\mu))^{\frac{1}{p-1}}\leq \delta_2\lambda\right\}\cap U_{\epsilon, \lambda} \cap \Omega \cap B_r(x). \label{EBr}
\end{align}

	 To prove \eqref{5hh2310133} we separately consider  the case $B_{8r}(x)\subset\subset\Omega$ and the case $B_{8r}(x)\cap\Omega^{c}\not=\emptyset$.
	
	\medskip
	
	 \noindent {\bf 1. The case $B_{8r}(x)\subset\subset\Omega$:} Applying  Proposition \ref{inter} to  $u=u_{k}\in W_{0}^{1,p}(\Omega),\mu=\mu_k$ and $B_{2R}=B_{8r}(x)$, there is a function $v_k\in W^{1,p}(B_{4r}(x))\cap W^{1,\infty}(B_{2r}(x))$ such that for any $\eta>0$, 
	 \begin{align*}
	& \|\nabla v_k\|_{L^\infty(B_{2r}(x))}\leq C_\eta \left[\frac{|\mu_k|(B_{8r}(x))}{r^{n-1}}\right]^{\frac{1}{p-1}}\\&~~~~~+C \left(\fint_{B_{8r}(x)}|\nabla u_k|^{\gamma_1}\right)^{\frac{1}{\gamma_1}}+\eta \left(\fint_{B_{8r}(x)}|\nabla u_k|^{2-p}\right)^{\frac{1}{2-p}},
	 \end{align*}
	 and
	 \begin{align*}
	 &\left(\fint_{B_{4r}}|\nabla u_k-\nabla v_k|^{\gamma_1}dx\right)^{\frac{1}{\gamma_1}}
         \leq C_\eta \left[\frac{|\mu_k|(B_{8r}(x))}{r^{n-1}}\right]^{\frac{1}{p-1}}\\&\quad\quad+C(([A]_{R_0})^{\kappa} +\eta)\left(\fint_{B_{8r}}|\nabla u_k|^{\gamma_1}\right)^{\frac{1}{\gamma_1}}+\eta \left(\fint_{B_{8r}(x)}|\nabla u_k|^{2-p}\right)^{\frac{1}{2-p}},
	 \end{align*}
	 for some $\kappa\in (0,1)$.
	
	 Notice that using \eqref{Mx1}, \eqref{MM1x2}, and property \eqref{grad-appr},
			 we get 
	 \begin{align*}
	 &\mathop {\limsup }\limits_{k \to \infty } \|\nabla v_{k}\|_{L^\infty(B_{2r}(x))}\\&\leq   C_\eta \left[\frac{|\mu|(\overline{B_{8r}(x)})}{r^{n-1}}\right]^{\frac{1}{p-1}}+C \left(\fint_{B_{8r}(x)}|\nabla u|^{\gamma_1}\right)^{\frac{1}{\gamma_1}} +\eta \left(\fint_{B_{8r}(x)}|\nabla u|^{2-p}\right)^{\frac{1}{2-p}}   \\&\leq   C_\eta [\mathbf{M}_1(\mu)(x_2)]^{\frac{1}{p-1}}+C \left[{\bf M}(|\nabla u|^{\gamma_1})(x_1)\right]^{\frac{1}{\gamma_1}}                     +C\eta \left[{\bf M}(|\nabla u|^{2-p})(x_2)\right]^{\frac{1}{2-p}}   
	 \\&\leq \left[C_\eta\delta_2+C+C\eta\epsilon^{-1}\right]\lambda\leq C_1\lam,
	 \end{align*}
	provided  $C_\eta\delta_2, \eta\epsilon^{-1}\leq 1$, and 
	 \begin{align*}
	 &\mathop {\limsup }\limits_{k \to \infty }  \left(\fint_{B_{4r}(x)}|\nabla u_k-\nabla v_k|^{\gamma_1}dx\right)^{\frac{1}{\gamma_1}}\leq C_\eta \left[\frac{|\mu|(\overline{B_{8r}(x)})}{r^{n-1}}\right]^{\frac{1}{p-1}}\\
	 &\qquad \qquad+C(([A]_{R_0})^{\kappa} +\eta)      \left(\fint_{B_{8r}(x)}|\nabla u|^{\gamma_1}\right)^{\frac{1}{\gamma_1}}+\eta \left(\fint_{B_{8r}(x)}|\nabla u|^{2-p}\right)^{\frac{1}{2-p}} \\
	 &\qquad\leq C_\eta [\mathbf{M}_1(\mu)(x_2)]^{\frac{1}{p-1}}+C(([A]_{R_0})^{\kappa} +\eta) \left[{\bf M}(|\nabla u|^{\gamma_1})(x_1)
         \right]^{\frac{1}{\gamma_1}}  \\
      &\qquad \qquad          +C\eta\left[{\bf M}(|\nabla u|^{2-p})(x_2)
         \right]^{\frac{1}{2-p}}   
	 \\&\qquad\leq C\left(C_{\eta}\delta_2+\delta_1^{\kappa}+\eta\varepsilon^{-1}\right)\lambda.
	 \end{align*}                                                                
	Here  we also used that $\mu_k\rightarrow \mu$ in the narrow topology of measures
	and  	$[A]_{R_0}\leq \delta_1$.

	 Thus there exists $k_0>1$ such that for all $k\geq k_0$ we have  
	 \begin{equation}\label{5hh2310136}
	 \|\nabla v_{k}\|_{L^\infty(B_{2r}(x))}\leq 2 C_1\lambda,
	 \end{equation}                        
and
	 \begin{equation}\label{5hh2310137} 
	 \left(\fint_{B_{4r}(x)}|\nabla u_k-\nabla v_k|^{\gamma_1}dx\right)^{\frac{1}{\gamma_1}}\leq C\left(C_{\eta}\delta_2+\delta_1^{\kappa}+\eta\varepsilon^{-1}\right)\lambda.
	 \end{equation}      
	   
	   Note that by \eqref{EBr} we find
	  \begin{align}\nonumber
	  |E_{\lambda,\delta_2}\cap B_r(x)|&\leq   |\{{\bf M}\left(\chi_{B_{2r}(x)}|\nabla (u_k-v_k)|^{\gamma_1}\right)^{\frac{1}{\gamma_1}}>\Lambda_0\lambda/9\}\cap B_r(x)|
	  \\&\nonumber+ |\{{\bf M}\left(\chi_{B_{2r}(x)}|\nabla (u-u_k)|^{\gamma_1}\right)^{\frac{1}{\gamma_1}}>\Lambda_0\lambda/9\}\cap B_r(x)|\\&+
	  |\{{\bf M}\left(\chi_{B_{2r}(x)}|\nabla v_k|^{\gamma_1}\right)^{\frac{1}{\gamma_1}}>\Lambda_0\lambda/9\}\cap B_r(x)|.         \label{es18}                 
	  \end{align}  

 On the other hand,	 in view of \eqref{5hh2310136} we see that for $\Lambda_0\geq \max\{3^{n},20C_1\}$ ($C_1$ is the constant in \eqref{5hh2310136}) and $k\geq k_0$, it holds that 
$$
	 |\{{\bf M}\left(\chi_{B_{2r}(x)}|\nabla v_k|^{\gamma_1}\right)^{\frac{1}{\gamma_1}}>\Lambda_0\lambda/9\}\cap B_r(x)|=0.$$

Thus, we deduce from \eqref{5hh2310137} and \eqref{es18}  that for any $k\geq k_0$,
	 \begin{align*}
	 |E_{\lambda,\delta_2}\cap B_r(x)|&\leq \frac{C}{\lambda^{\gamma_1}}    \left[\int_{B_{2r}(x)}  |\nabla (u_k-v_k)|^{\gamma_1}+ \int_{B_{2r}(x)}  |\nabla (u-u_k)|^{\gamma_1} \right]    
	 \\&\leq \frac{C}{\lambda^{\gamma_1}}    \left[\left(C_{\eta}\delta_2+\delta_1^{\kappa}+\eta\varepsilon^{-1}\right)^{\gamma_1}\lambda^{\gamma_1}r^n+ \int_{B_{2r}(x)}  |\nabla (u-u_k)|^{\gamma_1} \right].
	 \end{align*}

	 Then letting $k\to\infty$  we get 
$$
	 |E_{\lambda,\delta_2}\cap B_r(x)|\leq C \left(C_{\eta}\delta_2+\delta_1^{\kappa}+\eta\varepsilon^{-1}\right)^{\gamma_1}|B_r(x)|.
$$

	 This gives  
	 \begin{align*}
	 w(E_{\lambda,\delta_2}\cap B_r(x))&\leq c\left(\frac{|E_{\lambda,\delta_2}\cap B_r(x) |}{|B_r(x)|}\right)^\nu w(B_r(x))
	 \\&\leq  c\left(C_{\eta}\delta_2+\delta_1^{\kappa}+\eta\varepsilon^{-1}\right)^{\gamma_1\nu} w(B_r(x))
	 \\&< \varepsilon w(B_r(x)),
	 \end{align*} 
	 where $\eta,\delta_1\leq C(n,p, \Lambda,\gamma_1,\epsilon,[w]_{\mathbf{A}_\infty})$ and
	$\delta_2\leq C(n,p, \Lambda,\gamma_1,\epsilon,[w]_{\mathbf{A}_\infty}, R/R_0)$.
	
	\medskip
	
	 \noindent {\bf 2. The case $B_{8r}(x)\cap\Omega^{c}\not=\emptyset$:} Let $x_3\in\partial \Omega$ such that $|x_3-x|=\text{dist}(x,\partial\Omega)$.  We have 
$$
	 B_{2r}(x)\subset B_{10r}(x_3)\subset B_{100r}(x_3)\subset B_{108r}(x)\subset B_{109r}(x_1),
$$
	 and 
$$
	 B_{100r}(x_3)\subset B_{108r}(x)\subset B_{109r}(x_2).$$

	  Applying  Proposition \ref{boundary} to  $u=u_{k}\in W_{0}^{1,p}(\Omega),\mu=\mu_k$ and $B_{10R}=B_{100r}(x_3)$, 	for any $\eta>0$ there exists $\delta_0=\delta_0(n,p,\Lambda,\eta)$ such that the following holds. If $\Omega$ is a $(\delta_0,R_0)$-Reifenberg flat domain, there exists a function $V_k\in W^{1,\infty}(B_{10r}(x_3))$  such that 
	  \begin{align*}
	  \|\nabla V_k\|_{L^\infty(B_{10r}(x_3))}&\leq C_\eta \left[\frac{|\mu_k|(B_{100r}(x_3))}{r^{n-1}}\right]^{\frac{1}{p-1}}+C \left(\fint_{B_{100r}(x_3)}|\nabla u_k|^{\gamma_1}\right)^{\frac{1}{\gamma_1}}\\&+\eta \left(\fint_{B_{100r}(x_3)}|\nabla u_k|^{2-p}\right)^{\frac{1}{2-p}},
	  \end{align*}
and		
	  \begin{align*}\nonumber &
	  \left(\fint_{B_{10r}(x_3)}|\nabla (u_k-V_k)|^{\gamma_1}dx\right)^{\frac{1}{\gamma_1}}\leq C_\eta \left[\frac{|\mu_k|(B_{100r}(x_3))}{r^{n-1}}\right]^{\frac{1}{p-1}}\\&+C(([A]_{R_0})^{\kappa} +\eta)\left(\fint_{B_{100r}(x_3)}|\nabla u_k|^{\gamma_1}\right)^{\frac{1}{\gamma_1}}+\eta \left(\fint_{B_{100r}(x_3)}|\nabla u_k|^{2-p}\right)^{\frac{1}{2-p}},
	  \end{align*}
	  for some  $\kappa\in (0,1)$.
As above, we also obtain 
	\begin{equation*}
	|E_{\lambda,\delta_2}\cap B_r(x)|\leq C \left(C_{\eta}\delta_2+\delta_1^{\kappa}+\eta\varepsilon^{-1}\right)^{\gamma_1}|B_r(x)|,
	\end{equation*}
and thus
        \begin{align*}
	 w(E_{\lambda,\delta_2}\cap B_r(x))&\leq c\left(\frac{|E_{\lambda,\delta_2}\cap B_r(x) |}{|B_r(x)|}\right)^\nu w(B_r(x))
	 \\&\leq  c\left(C_{\eta}\delta_2+\delta_1^{\kappa}+\eta\varepsilon^{-1}\right)^{\gamma_1\nu} w(B_r(x))
	 \\&< \varepsilon w(B_r(x)).
	 \end{align*} 
	 where $\eta,\delta_1\leq C(n,p,\Lambda,\gamma_1,\varepsilon,[w]_{\mathbf{A}_\infty})$ and 
	$\delta_2\leq C(n,p,\Lambda,\gamma_1,\varepsilon,[w]_{\mathbf{A}_\infty}, R/R_0)$.

Using \eqref{5hh2310131} and \eqref{2nd-check},	we can now apply Lemma \ref{5hhvitali2} with $E= E_{\lambda, \delta_2}$ and $F=F_\lambda$ to complete the proof of the theorem.                         
\end{proof}

\section{Proof of Theorem \ref{101120143}}\label{sec-4}

This section is devoted to the proof of Theorem \ref{101120143}. Our main tools are  good-$\lambda$ type bounds
obtained in Theorem  \ref{5hh23101312} and  stability results of renormalized solutions obtained in 
\cite[Theorem 3.2]{11DMOP} 

\medskip
 
\noindent \begin{proof}[Proof of Theorem \ref{101120143}] 
	Let $R_0>0$ and fix a number $\gamma_1\in \left(0,\frac{(p-1)n}{n-1}\right)$. Suppose for now that $u$ is a renormalized solution of \eqref{5hh070120148} such that 
	$|\nabla u|\in L^{q}_w(\Om)$, $q>2-p$.
	
	By Theorem \ref{5hh23101312}, for any $\varepsilon>0,R_0>0$ one can find $\delta=\delta(n,p,\Lambda,\varepsilon,[w]_{{\bf A}_\infty})\in (0,1/2)$, $\delta_2=\delta_2(n,p,\Lambda,\varepsilon,[w]_{{\bf A}_\infty},diam(\Omega)/R_0)\in (0,1)$, and $\Lambda_0=\Lambda_0(n,p,\Lambda)>1$ such that if $\Omega$ is  a $(\delta,R_0)$-Reifenberg flat domain and $[A]_{R_0}\le \delta$ then 
	\begin{align*}
	&w(\{( {\bf M}(|\nabla u|^{\gamma_1}))^{\frac{1}{\gamma_1}}>\Lambda_0\lambda,{\bf M}(|\nabla u|^{2-p}))^{\frac{1}{2-p}}\leq \varepsilon^{-1}\lambda, (\mathbf{M}_1(\mu))^{\frac{1}{p-1}}\le \delta_2\lambda \}\cap \Omega) \nonumber\\
	&\qquad\leq C\varepsilon w(\{ ({\bf M}(|\nabla u|^{\gamma_1}))^{1/\gamma_1}> \lambda\}\cap \Omega), 
	\end{align*}
	for all $\lambda>0$. Here  the constant $C$  depends only on $n,p,\Lambda,[w]_{{\bf A}_\infty}$, and $diam(\Omega)/R_0$.

Thus, we find 	
\begin{align}
	&w(\{({\bf M}(|\nabla u|^{\gamma_1}))^{\frac{1}{\gamma_1}}>t\}\cap \Omega)\le 	w(\{ (\mathbf{M}_1(\mu))^{\frac{1}{p-1}}> \frac{\delta_2}{\Lambda_0} t \}\cap \Omega)\nonumber
	\\&+w(\{({\bf M}(|\nabla u|^{2-p}))^{\frac{1}{2-p}}>\frac{\varepsilon^{-1}}{\Lambda_0} t\}\cap \Omega)+ C\varepsilon w(\{ ({\bf M}(|\nabla u|^{\gamma_1}))^{\frac{1}{\gamma_1}}> \frac{t}{\Lambda_0}\}\cap \Omega) \label{lamepsi}
\end{align}
for all $t>0$. This gives, 
\begin{align*}
 & \|{\bf M}(|\nabla u|^{\gamma_1}))^{1/\gamma_1}\|_{L^{q}_w(\Omega)}\leq C\delta_2^{-1}   \|(\mathbf{M}_1(\mu))^{\frac{1}{p-1}}\|_{L^{q}_w(\Omega)}\\&+C\varepsilon  \|({\bf M}(|\nabla u|^{2-p}))^{1/(2-p)}\|_{L^{q}_w(\Omega)}+C\varepsilon  \|({\bf M}(|\nabla u|^{\gamma_1}))^{1/\gamma_1}\|_{L^{q}_w(\Omega)}.
\end{align*}

Using the boundedness of  ${\bf M}$ on  $L^{q/(2-p)}_w(\mathbb{R}^n)$, where $q/(2-p)>1$ and $w\in A_{\frac{q}{2-p}}$, and choosing $\epsilon<\frac{1}{2C}$ in the last inequality we deduce
$$\|{\bf M}(|\nabla u|^{\gamma_1}))^{1/\gamma_1}\|_{L^{q}_w(\Omega)}\leq 2C\delta_2^{-1}   \|(\mathbf{M}_1(\mu))^{\frac{1}{p-1}}\|_{L^{q}_w(\Omega)}+C'\varepsilon  \|\nabla u\|_{L^{q}_w(\Omega)}.
$$

Thus with $\varepsilon=\frac{1}{4(C+C')}$ we conclude that 
\begin{equation}\label{apbound}
 \|\nabla u\|_{L^{q}_w(\Omega)}\leq \widetilde{C} \ \|(\mathbf{M}_1(\mu))^{\frac{1}{p-1}}\|_{L^{q}_w(\Omega)}.
\end{equation}

To show existence, let $B_{R_1}(x_1)$ be a ball such that $\Om\Subset B_{R_1}(x_1)$ and extend $\mu$ by zero outside $\Om$. Then we can write 
$$\mu= f- {\rm div} F + \mu_s^{+} - \mu_s^{-},$$ 
as distributions in $B_{R_1+1}(x_1)$, where $f\in L^1(B_{R_1+1}(x_1))$, $F\in L^{\frac{p}{p-1}}(B_{R_1+1}(x_1),\RR^n)$, and $\mu_s$ is concentrated on a set of zero 
$p$-capacity. Let $\rho_\epsilon(x)=\epsilon^{-n} \rho(x/\epsilon)$ where $\rho\in C_c^\infty(B_1(0))$ is a nonnegative radial function with
$\|\rho\|_{L^1(\RR^n)}=1$. Then for any $\epsilon\in (0,1)$ we have 
$$\rho_\epsilon*\mu= \rho_\epsilon*f- {\rm div} (\rho_\epsilon*F) + \rho_\epsilon*\mu_s^{+} - \rho_\epsilon*\mu_s^{-}$$ 
as distributions in $B_{R_1}(x_1)$.

Let $u_\epsilon\in W^{1,p}_{0}(\Omega)$ be the unique solution of 
\begin{eqnarray*}
\left\{ \begin{array}{rcl}
-{\rm div}(A(x, \nabla u))&=& \rho_\epsilon*\mu \quad \text{in} ~\Omega, \\
u&=&0  \quad \text{on}~ \partial \Omega.
\end{array}\right.
\end{eqnarray*} 																																						

Then we can deduce from \cite[Theorem 1.10]{55MePh3} that $|\nabla u_\epsilon|\in L^{q}_w(\Om)$ provided $\delta=\delta(n,p,\Lambda, q,s, [w]_{\mathbf{A}_{\frac{q}{2-p}}})$ is sufficiently small.  Thus we may apply \eqref{apbound} and get 
\begin{align*}
\|\nabla u_\epsilon\|_{L^{q}_w(\Omega)} & \leq C  \|(\mathbf{M}_1(\rho_\epsilon*\mu))^{\frac{1}{p-1}}\|_{L^{q}_w(\Omega)}\\
& \leq C \|\left(\mathbf{M}[\mathbf{M}_1(\mu)]\right)^{\frac{1}{p-1}}\|_{L^{q}_w(\Omega)}\\
& \leq C \|(\mathbf{M}_1(\mu)^{\frac{1}{p-1}}\|_{L^{q}_w(\Omega)}.
\end{align*}

The theorem now follows from the stability result of  \cite[Theorem 3.2]{11DMOP}. 
\end{proof}

\section{Proof of Theorem \ref{main-Ric}}\label{sec-5}

We will need the following important compactness result.

\begin{lemma}\label{compactness} Suppose that  $1<p\leq \frac{3n-2}{2n-1}$. For each $j>0$, let $\mu_j \in \mathfrak{M}_0(\Omega)$ and $u_j$ be the solution of \eqref{5hh070120148} with datum $\mu=\mu_j$ in $\Omega$.
	Assume that 
	$\{\left[\mathbf{M}_1(\mu_j)\right]^{\frac{q}{p-1}}\}_{j}$, $q>2-p$, is a bounded and equi-integrable subset of  $L^1(\Omega)$.
	Then, there exists  $\delta=\delta(n,p,\Lambda, q)\in (0,1)$ such that if $\Omega$ is  $(\delta,R_0)$-Reifenberg flat  and $[A]_{R_0}\le \delta$ for some $R_0>0$, then there exist a subsequence 
	$\{u_{j'}\}_{j'}$ and a finite a.e. function $u$ with the property that $T_k(u)\in W^{1,p}_0(\Omega)$ for all $k>0$, $u_{j'}\rightarrow u$ a.e., and 
	\begin{equation}\label{strong-q-w}
	\nabla u_{j'} \rightarrow \nabla u \quad \text{strongly in} \quad L^q(\Omega, \mathbb{R}^n).
	\end{equation}
\end{lemma} 
 \begin{proof} By de la Vall\'ee-Poussin Lemma on equi-integrability, there exists  a strictly increasing and convex  function 
	$G:[0,\infty)\rightarrow [0,\infty)$ with $G(0)=0$ such that $\lim_{t\rightarrow \infty} G(t)/t=\infty$
	and
$$
	\sup_{j} \int_{\Omega} G([{\bf M}_1(|\mu_j|)]^{\frac{q}{p-1}}) w dx \leq C.
$$
	Moreover, we may assume that $G$ satisfies a moderate growth condition (see \cite{Mey}): there exists $c_1>1$ such that 
	$$G(2t)\leq c_1\, G(t)\qquad \forall t\geq 0.$$ 
	
	Let $\Phi(t):=G(t^q)$, where $q>2-p>p-1$. 	Then applying \eqref{lamepsi} with $w=1$ and with $\Phi^{-1}(t)$ in place of $t$ we find	
	\begin{align*}
	\left|\{\Phi[({\bf M}(|\nabla u_j|^{\gamma_1}))^{\frac{1}{\gamma_1}}]>t\}\cap \Omega\right|& \le 	\left|\{ \Phi[\frac{\Lambda_0}{\delta_2}(\mathbf{M}_1(\mu_j))^{\frac{1}{p-1}}]>  t \}\cap \Omega\right|\\
	& \quad + \left|\{ \Phi[\varepsilon \Lambda_0({\bf M}(|\nabla u_j|^{2-p}))^{\frac{1}{2-p}}]> t\}\cap \Omega\right|\\
	&  \quad + C\varepsilon \left| \{  \Phi[ \Lambda_0 ({\bf M}(|\nabla u_j|^{\gamma_1}))^{\frac{1}{\gamma_1}}]> t\}\cap \Omega\right|
	\end{align*}
	for any $\epsilon>0$. Here $\delta_2$ depends on $\epsilon$, but $\Lambda_0$ and $C$ do not.
	
	Then arguing as in the proof of \cite[Theorem 1.4]{QH4} we find
	\begin{align}
	\int_{\Om} \Phi[(\mathbf{M}(|\nabla u_j|^{\gamma_1}))^{\frac{1}{\gamma_1}}] dx &\leq H(\epsilon) \int_{\Om} \Phi[\mathbf{M}_{1}(\mu_j)^{\frac{1}{p-1}}] dx\nonumber\\
	& \qquad + 2  \int_{\Om} \Phi[ \epsilon \Lambda_0(\mathbf{M}(|\nabla u_j|^{2-p}))^{\frac{1}{2-p}}] dx \label{mot}
	\end{align}
	for all sufficiently small $\epsilon>0$.

	Note that  by approximation as in the proof of Theorem \ref{101120143}  and by uniqueness (see Remark \ref{uniq}), we may assume that $\mu_j\in C^\infty(\overline{\Omega})$. Thus by the result of  \cite[Theorem 1.10]{55MePh3}, we may assume that
	\begin{equation}\label{hai}
	\int_{\Om} \Phi(|\nabla u_j|) dx <+\infty.
	\end{equation}
	 
	Now as the function $t\mapsto \Phi(t^{\frac{1}{2-p}})=G(t^{\frac{q}{2-p}})$ satisfies the $\nabla_2$ condition (see \cite{Ph1-1}), we deduce (see, e.g., \cite{Ga, BK}) that 
	\begin{equation}\label{ba}
	2 \int_{\Om} \Phi[ \epsilon \Lambda_0(\mathbf{M}(|\nabla u_j|^{2-p}))^{\frac{1}{2-p}}] dx\leq C \Lambda_0 \epsilon  
	\int_{\Om} \Phi(|\nabla u_j|) dx.
	\end{equation}
	
	Combining \eqref{mot}, \eqref{hai}, \eqref{ba}, and choosing $\epsilon$ sufficiently small we arrive at 
$$
	\int_{\Omega}G(|\nabla u_j|^q)dx \leq C \int_{\Omega} G([\mathbf{M}_1(\mu_j)]^{\frac{q}{p-1}})dx\leq C.
$$
	
	Thus by de la Vall\'ee-Poussin Lemma the set   $\{|\nabla u_j|^q\}_j$ is  also bounded and equi-integrable in   $L^1(\Omega)$.
	
	On the other hand,  it follows from the proof of  \cite[Theorem 3.4]{11DMOP} that there exists a subsequence $\{u_{j'}\}_{j'}$ converging a.e. to  a  function $u$ such that $|u|<\infty$ a.e.,  $T_{k}(u)\in W^{1, p}_0(\Omega)$ for all $k>0$, and  moreover
$$\nabla u_{j'} \rightarrow \nabla u \quad {\rm a.e.~ in~} \Omega.
$$
	
	At this point,  applying  Vitali Convergence Theorem we obtain the strong convergence \eqref{strong-q-w} as desired.  
\end{proof}

\vspace{.2in}

\noindent \begin{proof}[Proof of Theorem \ref{main-Ric}] The proof of Theorem \ref{main-Ric} is based on Schauder Fixed Point Theorem using Lemma \ref{compactness} and Theorem  \ref{101120143}. Indeed, with these results at hand, the proof is similar to that of \cite[Theorem 1.9]{QH4}, and thus we  omit the details. 

\end{proof}	

\vspace{.2in}

\noindent {\large \bf Acknowledgment.} Quoc-Hung Nguyen is  supported  by the ShanghaiTech University startup fund. Nguyen Cong Phuc is supported in part by Simons Foundation, award number 426071.


\begin{thebibliography}{99}
%
%
%
%
%
%
%


\bibitem{bebo} P. Benilan, L. Boccardo, T. Gallouet, R. Gariepy, M. Pierre , and J. L. Vazquez,
	{\em An $L^1$ theory of existence and uniqueness of solutions of nonlinear elliptic equations,} Ann. Scuola Norm. Sup. Pisa 
	(IV) {\bf 22} (1995), 241--273.

%
%
%

\bibitem{BK} S. Bloom and R.  Kerman,  {\em Weighted Orlicz space integral inequalities for the Hardy-Littlewood maximal operator}, Studia Math. {\bf 110} (1994), 149--167. 


\bibitem{BGO} L. Boccardo, T.  Gallou\"et, and L. Orsina, {\it Existence and uniqueness of entropy solutions for nonlinear elliptic equations with measure data}, Ann. Inst. H. Poincar\'e Anal. Non Lin\'eaire {\bf 13} (1996),  539--551. 


%
%
%
%
\bibitem{VHV} M. F. Bidaut-Veron, M. Garcia-Huidobro, and L. Veron, {\it Remarks on some quasilinear equations with gradient terms and measure data}.  Recent trends in nonlinear partial differential equations. II. Stationary problems, 31--53, Contemp. Math., 595, Amer. Math. Soc., Providence, RI, 2013. 
%
%
%

\bibitem{BW2} S.-S. Byun and L. Wang, {\it Elliptic equations with BMO nonlinearity in Reifenberg domains},  Adv. Math.  {\bf 219}  (2008), 1937--1971.


%
%
%


\bibitem {11DMOP} G. Dal Maso, F. Murat, L. Orsina, and A. Prignet, {\em Renormalized solutions of elliptic equations with general measure data},  Ann. Scuola Norm. Super. Pisa (IV) {\bf 28} (1999), 741--808.









\bibitem{55DuzaMing} F. Duzaar and G. Mingione, {\em Gradient estimates via non-linear potentials},  Amer. J. Math.
 {\bf 133} (2011), 1093--1149. 



\bibitem{Duzamin2} F. Duzaar and G. Mingione, {\em Gradient estimates via linear and nonlinear potentials}, 
J. Funt. Anal. {\bf 259} (2010), 2961--2998.

%
%
%
%
%
%
%
%

\bibitem{FS} M. Fukushima, K. Sato, and S. Taniguchi, {\it On the closable part of pre-Dirichlet forms and the
fine support of the underlying measures}, Osaka J. Math. {\bf 28} (1991), 517--535.

\bibitem{Ga} D. Gallardo, {\em Orlicz spaces for which the Hardy-Littlewood maximal operator is bounded},
Publ. Mat. {\bf 32} (1988), 261--266.

%
%
%
%
%
%
%

\bibitem{HMV} K. Hansson, V. G. Maz'ya, and I. E. Verbitsky, {\it Criteria of solvability for
multidimensional Riccati equations},  Ark. Mat. {\bf 37}  (1999), 87--120.





%
%
%
%
%
%
%
%
%
%
%

\bibitem{KuMi12} T. Kuusi and G. Mingione,{\it  Universal potential estimates.} J. Funct. Anal. {\bf 262} (2012), 4205--4269.



\bibitem{55MePh2} T. Mengesha and  N. C. Phuc, {\em Weighted and regularity estimates for nonlinear equations on Reifenberg flat domains},  J. Differential Equations  {\bf 250} (2011), 1485--2507.

\bibitem{55MePh3} T. Mengesha and N. C. Phuc, {\em Quasilinear Riccati type equations with distributional data in Morrey space framework}, J.  Differential Equations {\bf 260} (2016), 5421--5449.




\bibitem{Mey} P.-A. Meyer, {\it Sur le lemme de la Vallée Poussin et un th\'eor\`eme de Bismut}, (French) S\'eminaire de Probabilit\'es, XII (Univ. Strasbourg, Strasbourg, 1976/1977), pp. 770--774, Lecture Notes in Math., {\bf 649}, Springer, Berlin, 1978.

\bibitem{Mi2} G. Mingione, {\em The Calder\'on-Zygmund theory for elliptic problems with measure data}, Ann. Scu. Norm. Sup. Pisa Cl. Sci. (5) {\bf 6} (2007), 195--261.

%
%
%

 \bibitem{55QH2} Q.-H. Nguyen, {\em Potential estimates and quasilinear parabolic equations with measure data}, Submitted for publication, arXiv:1405.2587v2.


\bibitem{55QH3} Q.-H. Nguyen, {\em Global estimates for quasilinear parabolic equations on Reifenberg flat domains and its applications to  Riccati type parabolic equations with distributional data}, Calc.  Var. Partial Differential Equations  {\bf 54} (2015),  3927--3948.

	\bibitem{QH4} Q.-H. Nguyen and N. C. Phuc, {\em Good-$\lambda$  and Muckenhoupt-Wheeden type bounds in quasilinear measure datum problems, with applications},  Math. Ann. {\bf 374} (2019), 67--98.
	
\bibitem{Ph1} N. C. Phuc, {\it Quasilinear Riccati type equations with super-critical exponents,} Comm. Partial Differential Equations \textbf{35} (2010), 1958--1981.

\bibitem{Ph1-1} N. C. Phuc, {\it Erratum to:  Quasilinear Riccati type equations with super-critical exponents},   Comm. Partial Differential Equations {\bf 42} (2017), 1335--1341. 



\bibitem{55Ph2} N. C. Phuc, {\em Nonlinear Muckenhoupt-Wheeden type bounds on Reifenberg flat domains, with applications to quasilinear Riccati type equations,} Adv.  Math. {\bf 250} (2014), 387--419.

\bibitem{55Ph2-2} N. C. Phuc, {\em Corrigendum to: Nonlinear Muckenhoupt-Wheeden type bounds on Reifenberg flat domains, with applications to quasilinear Riccati type equations},
 Adv.  Math. {\bf 328} (2018),  1353--1359.




\bibitem{55Re} E. Reifenberg, {\em Solutions of the Plateau Problem for $m$-dimensional surfaces of varying topological type}, Acta Math. {\bf 104} (1960), 1--92.

\bibitem{Sa} D. Sarason, {\it Functions of vanishing mean oscillation}, Trans. Amer. Math. Soc.
{\bf 207} (1975), 391--405.

\bibitem{Tor} T. Toro, {\it Doubling and flatness: geometry of measures}, Notices Amer. Math. Soc. {\bf 44} (1997) 1087--1094.




\end{thebibliography}
\end{document}